\documentclass{hha}

\usepackage{combelow}
\usepackage{amssymb,amsfonts,amsmath,amscd,latexsym,amsthm}
\usepackage[english]{babel}
\usepackage{graphicx}
\usepackage{mathptmx}      % use Times fonts if available on your TeX system
\usepackage{geometry}
\usepackage[all]{xy}
\usepackage[active]{srcltx}
\usepackage{color}

\numberwithin{equation}{section}
\theoremstyle{definition}
\newtheorem{subsec}[theorem]{ \ }
\newcommand{\Ker}{\operatorname{Ker}}
\newcommand{\Ima}{\operatorname{Im}}
\newcommand{\Hom}{\operatorname{Hom}}
\newcommand{\Hrm}{\operatorname{H}}
\newcommand{\HS}{\operatorname{HS}}
\newcommand{\HH}{\operatorname{HH}}
\newcommand{\HHS}{\operatorname{HHS}}
\newcommand{\one}{\operatorname{\mathrm{\mathbf{1}}}}
\newcommand{\id}{\operatorname{id}}
\newcommand{\Z}{\operatorname{\mathbb{Z}}}
\begin{document}

% Title of document, usually lower case except for first word
% and proper nouns.  Avoid unnecessary symbols.
\title{Symmetric Hochschild cohomology of twisted group algebras}

% If the title is too long for the running head, use
% the following command to specify a short title:
%\shorttitle{Shorter title}

% First Author
\author{Tiberiu Cocone\cb t}

% First Author's email address.  Must come before \address.
\email{tiberiu.coconet@math.ubbcluj.ro}

% First author's postal address, with *line breaks* like below.  Do
% not use \\ to separate lines.  It will appear on one line in the
% article, but will be used exactly as typed below to send a
% complimentary copy of the journal, so ensure that it is a complete
% postal address with line breaks as would appear on an envelope.
%
% Also, if it is not obvious, please add a comment with % that says
% what part is a district within a city, what part is a city,
% what part is a region or province and what part is a postal code.
%
% We recommend using the most current address possible.
% If you want the journal copy sent to a different address than
% the one below, please let Dan Christensen <jdc@uwo.ca> know.
\address{Babe\c s-Bolyai University,
Faculty of Economics and Business Administration,
 Str. Teodor Mihali, nr. 58-60,
 % postalcode 
 RO-400591 Cluj-Napoca,
 Romania.
%second adress (second afiliation) of the first author 
 Department of Mathematics,
Technical University of Cluj-Napoca,
Str. G. Baritiu 25,
 Cluj-Napoca  % postalcode 
 400027, Romania.}

% If needed, use a \thanks command, but not inside the author command.
 \thanks{This work was supported by a grant of the Romanian Ministry of Education and Research, CNCS - UEFISCDI, project number PN-III-P1-1.1-TE-2019-0136, within PNCDI III.}

% Second Author.
\author{Constantin-Cosmin Todea}
% Second Author's email address.  Must come before \address.
\email{Constantin.Todea@math.utcluj.ro}
% Second author's postal address.  Please read the instructions
% given above.  This one gives an example of an American address.
\address{Department of Mathematics,
Technical University of Cluj-Napoca,
Str. G. Baritiu 25,
 Cluj-Napoca % postalcode 
 400027, Romania.}

% Additional authors done in the same way.

% If the author names are too long for the running head, use
% the following command to specify a shorter version:
%\shortauthors{DOE, SMITH \andname\ WILLIAMS}

% AMS 2010 Mathematics Subject Classification.  List one or several,
% separated by commas, ending in a period.
% See http://www.ams.org/mathscinet/msc/msc2010.html for the 2010
% numbering system.
% Use \classification[2000]{12X34, 55X78.} if you must use codes
% from the 2000 numbering system.

% Keywords of the article, usually singular, no leading caps.
% Separated by commas, ending with period.

% Abstract comes before maketitle
\begin{abstract}
We show that there is an action of the symmetric group on the Hochschild cochain complex of a twisted group algebra with coefficients in a bimodule. This allows us to define the symmetric  Hochschild cohomology of twisted group algebras, similarly to Staic's construction of symmetric group cohomology. We give explicit embeddings and connecting homomorphisms between the symmetric cohomology spaces and symmetric Hochschild cohomology of twisted group algebras.
% Abstract text, usually no more than 200 words.
% Avoid bibliographic references (\cite) and complicated mathematics.
% Please do not use custom macros here, as this abstract has to
% be able to stand alone.  You may use standard tex/latex/AMS macros.
\end{abstract}

% Leave these items like this, and the journal will fill them in.
%\received{Month Day, Year}   % receive date (for example: October 11, 1999)
%\revised{Month Day, Year}    % date of revision; omit, if no revision;
                             % if multiple revisions, separate by commas
%\published{Month Day, Year}  % publish date
%\submitted{Bill Murray}      % Name of Journal's Editor, who handled Article
%\volumeyear{2012} % Volume Year
%\volumenumber{14} % Volume Number
%\issuenumber{1}   % Issue Number
%\startpage{1}     % PageNumber of first page
%\articlenumber{1} % Sequence number of article within issue
% If copyright is retained by author, comment this out:
%\owner{International Press}

\maketitle

% Text of Document.  Use constructs such as \section, \subsection,
% \begin{theorem} ... \end{theorem}, \begin{proof} ... \end{proof}, etc.

\section{Introduction}\label{sec1}

With a topological motivation in mind, Staic defined in \cite{St1} the symmetric cohomology of groups by constructing an action of the symmetric
group $\Sigma_{n+1}$ on $C^n(G,M)$ (here $G$ is a group, $M$ a $G$-module and $n$ a nonnegative integer). Since this action gives a subcomplex
$CS^n(G,M)=(C^n(G,M))^{\Sigma_{n+1}}$, of the classical group cohomology cochain complex, Staic obtained a new cohomology theory $\HS^n(G,M)$, called the symmetric cohomology
and a natural map
$\HS^n(G,M)\rightarrow \Hrm^n(G,M)$. Further, in \cite{St2} the author  managed to characterize the extensions of $G$ by $M$ which correspond to $\HS^2(G,M)$.
In \cite {Si1} Singh provides a similar construction for the topological groups and for  the Lie groups. Under some restrictions we were able  to prove in \cite{To} that the symmetric cohomology of groups has the structure of a Mackey functor. There is a growing interest in the study of symmetric cohomology in the last years. We mention \cite{Pi1}, where Pirashvili investigates the relation with the so-called exterior cohomology; we also mention \cite{BaNeSi} and \cite{Pi2}.

Let $G$ be a finite group, let $k$ be a field and consider $\alpha\in Z^2(G, k^{\times}),$ a $2$-cocycle. Here $k^{\times}$ is the multiplicative group of units in $k$. Let $k_{\alpha}G$ be the twisted group algebra corresponding to $[\alpha]\in\Hrm^2(G,k^{\times})$ and let $M$ be a $(k_{\alpha}G,k_{\alpha}G)$-bimodule. The Hochschild cohomology of algebras is an important invariant used in various settings. After reviewing, in Section \ref{sec2}, basic facts about Hochschild cohomology and group cohomology, we show in Section \ref{sec3} that there is a well-defined action (\ref{action_on_chain}) of the symmetric group $\Sigma_{*+1}$ on the Hochschild cochain complex
$C^*(k_{\alpha}G,M).$ By considering  the invariants we obtain the symmetric Hochschild cochain complex
$$CS^*(k_{\alpha}G,M)=(C^*(k_{\alpha}G,M))^{\Sigma_{*+1}}.$$
Its homology is called symmetric Hochschild cohomology of $k_{\alpha}G$ with coefficients in $M$ and is denoted $\HHS^*(k_{\alpha}G,M)$. These results are presented in
Proposition  \ref{prop32} and Definition \ref{defn33}.  We choose $\mathcal{B}=\{\bar{g}\mid g\in G\},$ a $k$-basis of the twisted group algebra $k_{\alpha}G$ and, we denote by $"\cdot "$ the multiplication in this twisted group algebra which is extended $k$-linearly from $\bar{g}\cdot \bar{h}=\alpha(g,h)\overline{gh},$
  for all $g,h\in G.$ The same notation is used for the right or for the left action of the twisted group algebra on bimodules.

For group algebras, the Hochschild cohomology admits an additive centralizers group cohomology decomposition. This decomposition was first introduced by  Burghelea \cite{Bu}, see \cite{LiZh} and \cite{SiWi} for more details with respect to this decomposition. A similar additive decomposition for $\HH^n(k_{\alpha}G)$ can be given in terms of some group cohomology spaces (with non-trivial coefficients), indexed by a system of representatives of the conjugacy classes of $G$, which we denote by $X$. Based on this decomposition we explicit two $k$-linear maps between Hochschild cohomology of  our twisted group algebra and these group cohomology spaces. These maps will be denoted $\nu_{G,x,\alpha}^*$, with $x\in Z(G)$; and $\pi_{G,x,\alpha}^*$ where $x$ runs in $X$, see Section \ref{sec4}.    We denote by $\one_k$ the trivial 2-cocycle, that is $\one_k(x,y)=1_k$ for all $x,y\in G.$   If $\alpha=\textbf{1}_k$ the above maps are denoted $\pi_{G,x}^*$ and $\nu_{G,x}^*$, respectively.  In this case both maps can be defined for any $x\in X$, as we can notice from \ref{5.3}. If $\alpha=\textbf{1}_k$ then  $k\bar{x}$ is denoted  $kx$  and is clear  that it can be identified to $k$ as trivial $kC_G(x)$-module.

 In Section \ref{sec5} we will show that these maps can be defined between the symmetric cohomology of groups and the symmetric Hochschild cohomology of twisted group algebras. We denote these maps by $\nu_{G,x,\alpha}^{*,S},$ $\pi_{G,x,\alpha}^{*,S}$ and, $\pi_{G,x}^{*,S},\nu_{G,x}^{*,S}$, respectively.  One of the proposed goals  of the paper was to obtain an additive decomposition for the symmetric Hochschild cohomology of twisted group algebras. In Section \ref{sec5} we prove our first main result,  Theorem \ref{thm11}. In  statement a) of this theorem we describe explicitly an embedding of additive decomposition (indexed by all elements of $Z(G)$)  of symmetric cohomology spaces (with non-trivial coefficients) into the symmetric Hochschild cohomology of twisted group algebras. Statement b) is about the similar embedding applied to  the case $\alpha=\textbf{1}_k$; but, in this situation  the additive decomposition is indexed by all elements of $G$ and, the symmetric cohomology spaces are with trivial coefficients.

For the rest of the Introduction, of the second part of Section \ref{sec1'} and, throughout Section \ref{sec6} we allow  $k$ to be any commutative ring; also, we consider   $G$-modules. Theorem \ref{thm11}  statement b) and the definitions of $\nu_{G,x}^*,\pi_{G,x}^*,\nu_{G,x}^{*,S},\pi_{G,x}^{*,S}$ can be stated without any restrictions for any commutative ring $k$.

In \cite[Proposition 5.3]{Si1} Singh developed a method to construct long exact sequences for symmetric cohomology. However the construction of long exact sequences is made under quite strong restriction of short exact sequences which possess a symmetric section compatible with the actions, see \cite[Section 5]{Pi1}. Recall that, for $f:M\rightarrow N$  a surjective homomorhism of $G$-modules, we say that  a set map $s:N\rightarrow M$ is a symmetric section of $f$, if $(f\circ s)(n)=n, s(-n)=-s(n)$ for any $g\in G,n\in N$; it is compatible with the actions, if $s(gn)=gs(n),$ for any $g\in G,n\in N$.
In subsection \ref{subsec61} we recall an explicit description  of the connecting homomorphism which appears in the long exact sequence of classical group cohomology.  Proposition \ref{prop62}  gives an alternative, more explicit, definition of the connecting homomorphism which appears in the long exact sequence of symmetric cohomology. But, we are still forced to construct these long exact sequences for short exact sequences which admits a symmetric action compatible the actions. We denote these homomorphisms by $\beta_G^{*,S}$, see Proposition \ref{prop62},  where the group $G$ and the $G$-modules will be obvious in the respective context.

Starting with three commutative rings and with a short exact sequence of their abelian groups (which has a symmetric section), we are able to construct a short exact sequence of $G$-modules which admits a symmetric action compatible with the actions, see (\ref{eq61}), (\ref{eq62}). Now, the same machinery like in \ref{subsec61} and Propostion \ref{prop62}, allow us to obtain in the second main result of this paper, Theorem \ref{thm13}, a connecting homomorphism between symmetric Hochschild cohomology of some of group algebras. In Theorem \ref{thm13} we show that this connecting homomorphism, denoted $\mathbb{B}_G^*$ (see (\ref{eq63})), is compatible with $\beta_{C_G(x)}^{*,S}$ (defined in (\ref{eq65})) , through $\pi^{*,S}_{G,x}$, where $x\in X$.

We end Section \ref{sec6} with the proof of Proposition \ref{prop62}, Theorem \ref{thm13} and with further remarks. In these remarks we mention examples of a short exact sequences for which we can apply Theorem \ref{thm13}.

It is our strong belief that, although the modern approach of homological algebra is about derived functors and derived categories, explicit description of cochain complexes and various cohomology maps could have future applications. These methods, although technical and sometimes lengthier, should not be forgotten.
%Bockstein homomorphisms in group cohomology are the connecting homomorphism in the long exact sequence obtained from some short exact sequence
%of coefficients. It is a tool for comparing integral and mod-$p$ cohomology and has applications in Steenrod operations.
\section{Main results}\label{sec1'}
With the notations from the Introduction, we collect all the main results, explained in the first part of Introduction.
For  $k_{\alpha}G$ there is an additive decomposition similar to $\HH^*(kG)$, which can be recovered from \cite[Lemma 3.5]{Wi}:
  \begin{equation} \label{eq11}
  \HH^n(k_{\alpha}G)\cong  \bigoplus_{x\in X}\Hrm^n(C_G(x),k\bar{x})
  \end{equation}
where  $X$ is a system of representatives of the conjugacy classes of $G$ and $k\bar{x}$ is the $kC_G(x)$-module given by $$h\bar{x}=\alpha(h,x)(\alpha(x,h))^{-1}\bar{x},$$ for any $h\in C_G(x).$

For any choice of $X$, we have $Z(G)\subseteq X$ hence, we consider an embedding
\begin{equation} \label{eq12}
\bigoplus_{x\in Z(G)}\Hrm^n(C_G(x),k\bar{x})\rightarrowtail \HH^n(k_{\alpha}G).
\end{equation}
In Section \ref{sec4} we define some $k$-linear maps between cochain complexes which determine
$k$-linear maps, mentioned in Introduction, between the following cohomology spaces
$$\nu_{G,x,\alpha}^*:\Hrm^*(G,k\bar{x})\rightarrowtail \HH^*(k_{\alpha}G)$$
if $x\in Z(G)$ and,
$$\pi_{G,x,\alpha}^*:\HH^*(k_{\alpha}G)\twoheadrightarrow\Hrm^*(C_G(x),k\bar{x})$$
if $x\in X$ (see Proposition \ref{proposition4.6}).
The first theorem of this paper is the following.
\begin{theorem}\label{thm11}
Let $n\in\mathbb{Z},n\geq 0$ and $\alpha\in Z^2(G,k^{\times})$.
\begin{itemize}
\item[a)] There is an embedding of cohomology vector spaces
$$\bigoplus_{x\in Z(G)}\nu_{G,x,\alpha}^{n,S}:\bigoplus_{x\in Z(G)}\HS^n(G,k\bar{x})\rightarrowtail \HHS^n(k_{\alpha}G)$$
with its left inverse
$$\bigoplus_{x\in Z(G)}\pi_{G,x,\alpha}^{n,S}:\HHS^n(k_{\alpha}G)\twoheadrightarrow\bigoplus_{x\in Z(G)}\HS^n(G,k\bar{x}).$$
\item[b)] In the case $\alpha=\textbf{1}_k$, there is an embedding of cohomology vector spaces
$$\bigoplus_{x\in X}\nu_{G,x}^{n,S}:\bigoplus_{x\in X}\HS^n(C_G(x),k)\rightarrowtail \HHS^n(kG)$$
with its left inverse
$$\bigoplus_{x\in X}\pi_{G,x}^{n,S}:\HHS^n(kG)\twoheadrightarrow\bigoplus_{x\in X}\HS^n(C_G(x),k).$$
\end{itemize}
\end{theorem}
Recall that for the rest of section we allow to work with commutative ground rings and we use $G$-modules. 
The construction which we expose in the next subsection is considered folklore by experts, but we need to present the details for further application. Let $n\in\Z,n\geq 0$.
\begin{subsec}\textit{Connecting homomorphism in group cohomology, defined explicitly.}\label{subsec61}

Let
\begin{equation}\label{eq61'}
\xymatrix{0\ar[r]&A\ar[r]^-{\iota}
&B\ar[r]^-{\pi}&C\ar[r] &0}
\end{equation}
be a short exact sequence of $G$-modules. Since $\iota$ is injective,  the isomorphism of $kG$-modules $\iota_1:A\to \Ker \pi, \iota_1(a)=\iota(a)$ for any $a\in A$, has an inverse, which we denote by $r_{\iota}: \Ker \pi\rightarrow A$; hence $r_{\iota}$ is an isomoprhism of $kG$-modules, such that $(r_{\iota}\circ \iota)(a)=a$ for any $a\in A$ and, $(\iota\circ r_{\iota})(b)=b $ for any $b\in\Ker \pi$. Similarly, since $\pi$ is surjective, it has a section (right inverse set map) which we denote by $s_{\pi}:C\rightarrow B$. We will work with group cohomology  and various coefficient modules. For this reason we denote by $\partial_A$ the differential which gives $\Hrm^*(G,A),$ respectively $\partial _B,\partial_C$ etc.
 With the  notations above it is easy to prove that if $\varphi\in\Ker \partial_C^n$ then $\Ima\partial_B^n(s_{\pi}\circ \varphi)\subseteq \Ker \pi$ and $r_{\iota}\circ\partial_B^n(s_{\pi}\circ\varphi)\in\Ker\partial_A^{n+1}.$

By the Long Exact Sequence Theorem applied to (\ref{eq61'}) there is a connecting homomorphism, which we denote by
$$\beta_G^n:\Hrm^n(G,C)\rightarrow\Hrm^{n+1}(G,A).$$
The above statements allow us to define $\beta_G^n$ explicitly, as follows
\[\beta_G^n([\varphi]):=[\beta_G^n(\varphi)],\quad \beta_G^n(\varphi)=r_{\iota}\circ\partial_B^n(s_{\pi}\circ \varphi),\]
for any $\varphi\in\Ker\partial_C^n.$
\end{subsec}

\begin{proposition}\label{prop62} Let (\ref{eq61'}) be a short exact sequence of $G$-modules admitting a symmetric section compatible with the actions. Then there is a connecting homomorphism in symmetric cohomology
$$\beta_G^{n,S}:\HS^n(G,C)\rightarrow\HS^{n+1}(G,A)$$
given by $\beta_G^{n,S}([\varphi]):=[\beta_G^n(\varphi)]$ for any $\varphi\in\Ker\partial_C^n \cap CS^n(G,C).$
\end{proposition}
In order to obtain a connecting homomorphism in symmetric Hochschild cohomology we need the following setup. Let $k_1,k_2$ and $k_3$ be commutative rings and let
\begin{equation}\label{eq61}
\xymatrix{0\ar[r]&(k_1,+)\ar[r]^{\iota}&(k_2,+)\ar[r]^{\pi}&(k_3,+)\ar[r]&0}
\end{equation}
be a short exact sequence of abelian groups which admits a symmetric section $s_{\pi}:k_3\rightarrow k_2$ and  $r_{\iota}:\Ker\pi \rightarrow k_1$ constructed like in subsection \ref{subsec61}.
Next, there is a short exact sequence of $G$-modules
\begin{equation}\label{eq62}
\xymatrix{0\ar[r]&(k_1G,+)\ar[r]^{{\iota}_G}&(k_2G,+)\ar[r]^{\pi_G}&(k_3G,+)\ar[r]&0},
\end{equation}
where $G$ acts by conjugation and $\iota_G(a_1g)=\iota(a_1)g, \pi_G(a_2g)=\pi(a_2)g,$
for all $a_1\in k_1,a_2\in k_2,g\in G.$
It is not difficult to see that we can choose $r_{\iota_G}:\Ker \pi_G\rightarrow k_1G$, given by $r_{\iota_G}(\sum_{g\in G} a_gg)=\sum_{g\in G} r_{\iota}(a_g)g$ for all finite sums $\sum_{g\in G} a_gg \in \Ker\pi_G$. In the same way
$s_{\pi_G}:k_3G\rightarrow k_2G$ is a symmetric section of $\pi_G$ given by $s_{\pi_G}(a_3g)=s(a_3)g$ for all $a_3\in k_3,g\in G$. The maps $r_{\iota_G}$ and  $s_{\pi_G}$ are compatible with the left actions.  But, $r_{\iota_G},s_{\pi_G}$ are also compatible with the  right $G$-action and, they are homomorphisms of $(G-G)$-bimodules. The same arguments as in the proof of Proposition \ref{prop62} will assure us that the following homomorphism
\begin{equation}\label{eq63}
\begin{split}
\mathbb{B}_G^*:\HHS^*(k_3G)\rightarrow &\HHS^{*+1}(k_1G), \quad\mathbb{B}_G^*([\varphi]):=[\mathbb{B}_G^*(\varphi)], \\
&\mathbb{B}_G^*(\varphi)=r_{\iota_G}\circ\delta_{k_2G}^*(s_{\pi_G}\circ \varphi),
\end{split}
\end{equation}
for any $\varphi\in \mathrm{CS}^*(k_3G,k_3G)$, is a well-defined homomorphism; here $\delta_{k_2G}^*$ is the Hochschild differential of $\HH^*(k_2G)$, see  \ref{subsec21}.
This homomorphism  will be called the connecting homomorphism in symmetric Hochschild cohomology. Let $x\in X$ be a representative of conjugacy classes of $G$. We can also verify that the following short exact sequence of trivial $C_G(x)$-modules
\begin{equation}\label{eq64}
\xymatrix{0\ar[r]&(k_1x,+)\ar[r]^{\iota_x}&(k_2x,+)\ar[r]^{\pi_x}&(k_3x,+)\ar[r]&0}
\end{equation}
with $\iota_x(a_1x)=\iota(a_1)x, \pi_x(a_2x)=\pi(a_2)x,$
for all $a_1\in k_1,a_2\in k_2$, is a short exact sequence which possess a symmetric section $s_{\pi_x}$ compatible with the actions and,  an isomorphism $r_{\iota_x}$ defined as in subsection \ref{subsec61}; here $s_{\pi_x}(a_3x)=s_{\pi}(a_3)x$ for all $a_3\in k_3$ and similarly for $r_{\iota_x}$. We apply Proposition \ref{prop62} to (\ref{eq64}) to obtain a homomorphism
\begin{equation}\label{eq65}\beta_{C_G(x)}^{*,S}:\HS^*(C_G(x),k_3x)\rightarrow\HS^{*+1}(C_G(x),k_1x).
\end{equation}
The above maps fits in the second main result of this paper.

\begin{theorem}\label{thm13}
Let $n\in\Z, n\geq 0$. With the above notations consider a short exact sequence like in (\ref{eq61}). Then there exists a connecting homomorphism in symmetric Hochschild cohomology $\mathbb{B}_G^n:\HHS^n(k_3G)\rightarrow\HHS^{n+1}(k_1G),$ defined in (\ref{eq63}), making the following diagram commutative
  $$\xymatrix{ \HHS^n(k_3G)\ar[d]^{\pi^{n,S}_{G,x}}\ar[r]^{\mathbb{B}_G^n} &\HHS^{n+1}(k_1G)\ar[d]^{\pi^{n+1,S}_{G,x}} \\
   \HS^n(C_G(x),k_3x)\ar[r]^{\beta_{C_G(x)}^{n,S}}&\HS^{n+1}(C_G(x),k_1x).}$$
\end{theorem}
\section{Reminder on Hochschild cohomology and symmetric cohomology of groups}\label{sec2}
\begin{subsec}{\textit{The Hochschild cohomology.}}\label{subsec21}

With the above notations we know that Hochschild cohomology is
\[\HH^n(k_{\alpha}G,M)=\mathrm{H}^n(C^*(k_{\alpha}G,M))=\Ker  \delta^n/\Ima \delta^{n-1}, n>0\]
where
\[C^*(k_{\alpha}G,M)=\{C^n(k_{\alpha}G,M)\}_{n\geq 0} \mbox{ and } C^n(k_{\alpha}G,M)=\Hom_{k}((k_{\alpha}G)^{\otimes n},M), n>0.\]
If $n=0$ then $C^0(k_{\alpha}G,M)$ is identified with $M$.
The cochain complex $C^*(k_{\alpha}G,M)$ is given by:
\begin{equation}\xymatrix{&C^0(k_{\alpha}G,M)\ar[r]^{\delta^0}&\ldots \ar[r]&C^{n-1}(k_{\alpha}G,M)\ar[r]^{\delta^{n-1}}\ar[r]&C^n(k_{\alpha}G,M)\ar[r]^{\delta^n}\ar[r]&C^{n+1}(k_{\alpha}G,M)\ldots}
\end{equation}
\[ \delta^n:C^n(k_{\alpha}G,M)\to C^{n+1}(k_{\alpha}G,M)\]
\[(\varphi:(k_{\alpha}G)^{\otimes n}\to M )\mapsto (\delta^n(\varphi):(k_{\alpha}G)^{\otimes (n+1)}\to M)\]
\begin{align*}
\delta^n(\varphi)(a_1\otimes\ldots \otimes a_{n+1})&=a_1\cdot \varphi(a_2\otimes \ldots \otimes a_{n+1})+
\sum_{j=1}^n(-1)^j \varphi(a_1\otimes\ldots \otimes a_j\cdot a_{j+1}\otimes \ldots \otimes a_{n+1})\\
&+(-1)^{n+1} \varphi (a_1\otimes\ldots a_n)\cdot a_{n+1},
\end{align*}
for all $a_1, \ldots, a_{n+1}\in k_{\alpha}G$, if $n>0$.

In the case $n=0$ the differential $\delta_0:M\to C^{1}(k_{\alpha}G,M)$ is given by
$\delta_0(m)(a_1)=a_1\cdot m-m\cdot a_1$ for any $m\in M, a_1\in k_{\alpha}G$. It follows that
$$\HH^0(k_{\alpha}G,M)\cong\{m\in M | am=ma\  \text{for all}\ a\in A \}.$$
\end{subsec}

\begin{remark}\label{rem}
 In general, the following notations are frequently used:
 \begin{align*}
  \delta^n(\varphi)&=\sum_{j=0}^{n+1}(-1)^jd_j(\varphi),\\
  d_0(\varphi)(a_1\otimes \ldots \otimes a_{n+1})&=a_1\cdot \varphi(a_2\otimes \ldots\otimes a_{n+1});\\
  d_j(\varphi)(a_1\otimes \ldots \otimes a_{n+1})&=\varphi(a_1\otimes \otimes \ldots\otimes a_j\cdot a_{j+1}\otimes \ldots \otimes a_{n+1});\\
  d_{n+1}(\varphi)(a_1\otimes \ldots \otimes a_{n+1})&= \varphi(a_1\otimes \ldots\otimes a_{n})\cdot a_{n+1};
 \end{align*}
 for all $j\in \{1,\ldots, n\},$ $\varphi\in C^n(k_{\alpha}G,M)$ and $a_1, \ldots a_{n+1}\in k_{\alpha}G.$
\end{remark}

\begin{subsec}{\textit{ The symmetric cohomology of groups.}}

Let $n\in \Z,$ $n> 0.$ For any group $G$ and any $kG$-module $A,$ we recall the classical $n$-th cohomology of $G$ with coefficients in $A.$ Explicitly
\[\mathrm{H}^n(G,A)=\mathrm{H}^n(C^*(G,A))=\Ker \partial^n/\Ima \partial^{n-1},\]
where \[C^*(G,A)=\{C^n(G,A)\}_{n\geq  0} \mbox{ and } C^n(G,A)=Maps(G^{\times n},A).\]
The cochain complex $C^*(G,A)$ is given by
\begin{equation}\xymatrix{&C^0(G,A)\ar[r]^{\partial^0}&C^1(G,A)\ldots \ar[r]&C^{n-1}(G,A)\ar[r]^{\partial^{n-1}}\ar[r]&C^n(G,A)\ar[r]^{\partial^n}
\ar[r]&C^{n+1}(G,A)\ldots }
\end{equation}
\[ \partial^n:C^n(G,A)\to C^{n+1}(G,A)\]
\[(\varphi:G^{\times n}\to A )\mapsto (\partial^n(\varphi):G^{\times (n+1)}\to A)\]
\begin{align*}
\partial^n(\varphi)(g_1,\ldots , g_{n+1})&=g_1 \varphi(g_2, \ldots , g_{n+1})+
\sum_{j=1}^n(-1)^j \varphi(g_1, \ldots , g_jg_{j+1}, \ldots, g_{n+1})\\
&+(-1)^{n+1}\varphi (g_1,\ldots, g_n),
\end{align*}
for all $g_1, \ldots, g_{n+1}\in G.$

If $n=0$ then $C^0(G,A)=A$ and $\partial^0:C^0(G,A)\rightarrow C^1(G,A)$ is given by $\partial^0(a)(g)=ga-a$ for any $a\in A,g\in G$; thus $\mathrm{H}^0(G,A)=A^G$, the subgroup of $G$-invariants elements of $A$.

  In \cite[Section 5]{St1}  Staic defined an action of the symmetric group $\Sigma_{n+1}$ on $C^n(G,A)$ which is compatible with the differential $\partial^*;$ see \cite[Proposition 5.1]{St1}. The  subcomplex of $\Sigma_{n+1}$-invariant elements, denoted
  \[CS^n(G,A)=(C^n(G,A))^{\Sigma_{n+1}},\] is called the symmetric cochain complex. Its homology, denoted $\HS^n(G,A),$ is called the symmetric cohomology of $G$ with coefficients in $A;$ see \cite[Definition 5.2]{St1}.
\end{subsec}

\section{The symmetric Hochschild cochain complex of twisted group algebras}\label{sec3}

  Since $\alpha:G\times G\to k^{\times}$ is a $2$-cocycle, it is known that
  \begin{equation}\label{eq31}
  \alpha(x,yz)\alpha(y,z)=\alpha(xy,z)\alpha(x,y),
  \end{equation}
  for all $x,y,z\in G.$ In particular
  \begin{equation}\label{eq32}
  \alpha(x,x^{-1})=\alpha(x^{-1},x),\quad \alpha(x,1)=\alpha(1,x)=1_k
  \end{equation}
 and
 \begin{equation}\label{eq33}
 (\bar{g})^{-1}=(\alpha(g,g^{-1}))^{-1}\ \overline{g^{-1}}
 \end{equation}
for all $g\in G.$

\begin{subsec} \textit{An action of $\Sigma_{n+1}$ on $C^n(k_{\alpha}G,M)$.}

Recall the Coxeter presentation of the symmetric group on $n+1$ letters, $\Sigma_{n+1}:$
\begin{align*}\Sigma_{n+1}=\langle \sigma_i, i\in \{1,\ldots, n\}\mid &\sigma_i^2=e \mbox{ for } i\in \{1,\ldots n\},\\ &\sigma_i\sigma_{i+1}\sigma_i=\sigma_{i+1}\sigma_i\sigma_{i+1} \mbox{ for } i\in \{1,\ldots, n-1\}, \\ &\sigma_i\sigma_j=\sigma_j\sigma_i \mbox{ for } \mid i-j\mid \geq 2\rangle.\end{align*}
We define an action of $\Sigma_{n+1}$ on $C^n(k_{\alpha}G,M)$ by describing the action of the transpositions $\sigma_i=(i,i+1),$ $1\leq i\leq n.$

  Let $\varphi\in C^n(k_{\alpha}G,M).$ For each $i$ we define $\sigma_i\varphi\in C^n(k_{\alpha}G,M)$ using its restriction on the $k$-basis $\mathcal{B}^{\otimes n}$ of $(k_{\alpha}G)^{\otimes n}$ and, after this, we extend by $k$-linearity:

  \begin{equation} \label{action_on_chain}
  \begin{split}
  (\sigma_1\varphi)(\overline{g_1}\otimes\ldots \otimes \overline{g_n})&=-\overline{g}_1\cdot \varphi(\overline{g_1}^{-1}\otimes \overline{g_1}\cdot \overline{g_2}\otimes \ldots \otimes \overline{g_n}),\\
  (\sigma_i\varphi)(\overline{g_1}\otimes\ldots \otimes \overline{g_n})&=-\varphi(\overline{g_1}\otimes \ldots \otimes \overline{g_{i-2}}\otimes \overline{g_{i-1}}\cdot \overline{g_i}\otimes \overline{g_i}^{-1}\otimes \overline{g_i}\cdot \overline{g_{i+1}}\otimes \overline{g_{i+2}}\otimes \ldots\otimes \overline{g_n}) \mbox{ for } 1<i<n, \\
  (\sigma_n\varphi)(\overline{g_1}\otimes\ldots \otimes \overline{g_n})&=-\varphi(\overline{g_1}\otimes \overline{g_2}\otimes \ldots \otimes \overline{g_{n-1}}\cdot \overline{g_n}\otimes \overline{g_n}^{{}^{-1}})\cdot \overline{g_n},
  \end{split}
  \end{equation}
  for any $\overline{g_1},\ldots , \overline{g_n}\in \mathcal{B}.$
\end{subsec}

\begin{proposition}\label{prop32}
The above relations determine an action of $\Sigma_{n+1}$ on $C^n(k_{\alpha}G,M),$ which is compatible with the differential $\delta^*.$
\end{proposition}

\begin{proof}
Consider $\overline{g_1}, \ldots ,\overline{g_n}\in \mathcal{B}$ and $\varphi:(k_{\alpha}G)^{\otimes n}\to M$ a $k$-linear map.
We start with the square identity relations.
  \begin{align*}
   &(\sigma_1(\sigma_1 \varphi))(\overline{g_1}\otimes \ldots \otimes \overline{g_n})\\
   &=-\overline{g_1}\cdot (\sigma_1\varphi)(\overline{g_1}^{-1}\otimes \overline{g_1}\cdot \overline{g_2}\otimes \ldots \otimes \overline{g_n})\\
   &=-\alpha(g_1,g_1^{-1})^{-1}\alpha(g_1,g_2)\ \overline{g_1}\cdot (\sigma_1\varphi)(\overline{g_1^{-1}}\otimes \overline{g_1g_2}\otimes \ldots \otimes \overline{g_n})\\
   &=\alpha(g_1,g_2)\alpha(g_1,g_1^{-1})^{-1}\alpha(g_1,g_1^{-1})\alpha(g_1^{-1},g_1)^{-1}\alpha(g_1^{-1},g_1g_2)\varphi(\overline{g_1}\otimes \overline{g_2}\otimes \ldots \otimes \overline{g_n})\\
   &=\alpha(g_1^{-1},g_1g_2)\alpha(g_1,g_2)\alpha(g_1^{-1},g_1)^{-1}\varphi(\overline{g_1}\otimes \overline{g_2}\otimes \ldots \otimes \overline{g_n})\\
   &=\alpha(g_1^{-1}g_1,g_2)\alpha(g_1^{-1},g_1)\alpha(g_1^{-1},g_1)^{-1}\varphi(\overline{g_1}\otimes \overline{g_2}\otimes \ldots \otimes \overline{g_n})\\
   &=\varphi(\overline{g_1}\otimes \overline{g_2}\otimes \ldots \otimes \overline{g_n}).
  \end{align*}
Let $1<i<n,$ then
\begin{align*}
&(\sigma_i(\sigma_i \varphi))(\overline{g_1}\otimes \ldots \otimes \overline{g_n})\\
&=-(\sigma_i\varphi)(\overline{g_1}\otimes \ldots \otimes \overline{g_{i-2}}\otimes \overline{g_{i-1}}\cdot \overline{g_i}\otimes \overline{g_i}^{-1}\otimes \overline{g_i}\cdot \overline{g_{i+1}}\otimes \overline{g_{i+2}}\otimes \ldots\otimes \overline{g_n})\\
&=-\alpha(g_{i-1},g_i)\alpha(g_i,g_i^{-1})^{-1}\alpha(g_i,g_{i+1})(\sigma_i\varphi)\ (\overline{g_1}\otimes \ldots \otimes \overline{g_{i-2}}\otimes \overline{g_{i-1}g_i}\otimes \overline{g_i^{-1}}\otimes \overline{g_ig_{i+1}}\otimes \overline{g_{i+2}}\otimes \ldots\otimes \overline{g_n})\\
&=\alpha(g_{i-1},g_i)\alpha(g_{i-1}g_i,g_i^{-1})\alpha(g_i,g_{i+1})\alpha(g_i^{-1},g_ig_{i+1})\alpha(g_i,g_i^{-1})^{-1}\alpha(g_i^{-1},g_i)^{-1}
\varphi(\overline{g_1}\otimes \ldots  \otimes \overline{g_i}\otimes   \ldots\otimes \overline{g_n})\\
&=\alpha(g_{i-1},1)\alpha(g_i,g_i^{-1})\alpha(g_i^{-1},g_i)\alpha(1,g_{i+1})\alpha(g_i,g_i^{-1})^{-1}\alpha(g_i^{-1},g_i)^{-1}
\varphi(\overline{g_1}\otimes \ldots  \otimes \overline{g_i}\otimes   \ldots\otimes \overline{g_n})\\
&=\varphi(\overline{g_1}\otimes \ldots  \otimes \overline{g_i}\otimes   \ldots\otimes \overline{g_n}).
\end{align*}
Next we have
\begin{align*}
&(\sigma_n(\sigma_n\varphi))(\overline{g_1}\otimes \ldots    \ldots\otimes \overline{g_n})\\
&=-(\sigma_n\varphi)(\overline{g_1}\otimes \ldots  \otimes \overline{g_{n-1}}\cdot \overline{g_n}\otimes \overline{g_n}^{^{-1}})\cdot \overline{g_n}\\
&=-\alpha(g_{n-1},g_n)\alpha(g_n,g_n^{-1})^{-1}(\sigma_n\varphi)(\overline{g_1}\otimes \ldots  \otimes \overline{g_{n-1} g_n}\otimes \overline{g_n^{-1}})\cdot \overline{g_n}\\
&=\alpha(g_{n-1},g_n)\alpha(g_{n-1}g_n,g_n^{-1})\alpha(g_n,g_n^{-1})^{-1}\alpha(g_n^{-1},g_n)^{-1}\varphi(\overline{g_1}\otimes \ldots  \otimes \overline{g_{n-1} }\otimes \overline{g_n})\cdot \overline{g_n^{-1}}\cdot\overline{g_n}\\
&=\varphi(\overline{g_1}\otimes \ldots  \otimes \overline{g_{n-1} }\otimes \overline{g_n}).
\end{align*}
We continue with the cubic relations, that is $$\sigma_i(\sigma_{i+1}(\sigma_i\varphi))=\sigma_{i+1}(\sigma_{i}(\sigma_{i+1}\varphi)),$$ for $i\in \{1\ldots n-1\}.$ We only verify the case $i=n-1$, the other are similar,

\begin{align*}
&\sigma_{n-1}(\sigma_n(\sigma_{n-1}\varphi))(\overline{g_1}\otimes\ldots \overline{g_n})=-\sigma_{n}(\sigma_{n-1}\varphi)(\overline{g_1}\ldots\otimes\overline{g_{n-2}}\cdot \overline{g_{n-1}} \otimes \overline{g_{n-1}}^{^{-1}}\otimes\overline{g_{n-1}}\cdot \overline{g_n})\\
&=\alpha(g_{n-2},g_{n-1})\alpha(g_{n-1},g_{n-1}^{-1})^{-1}\alpha(g_{n-1},g_n),\\&(\sigma_{n-1}\varphi)
(\overline{g_1}\ldots\otimes\overline{g_{n-2}g_{n-1}} \otimes \overline{g_{n-1}^{-1}}\cdot \overline{g_{n-1}g_n} \otimes\overline{g_{n-1}g_n}^{^{-1}})\cdot \overline{g_{n-1}g_n}\\
&=\alpha(g_{n-2},g_{n-1})\alpha(g^{-1}_{n-1},g_{n-1}g_n)\alpha(g_{n-1},g_n)\alpha(g_{n-1},g_{n-1}^{-1})^{-1}\alpha(g_{n-1}g_{n},g_{n}^{-1}g_{n-1}^{-1})^{-1},\\
&(\sigma_{n-1}\varphi)(\overline{g_1}\otimes\ldots \otimes \overline{g_{n-2}g_{n-1}}\otimes\overline{g_n}\otimes\overline{g_{n}^{-1}g_{n-1}^{-1}})\cdot \overline{g_{n-1}g_n}\\
&=-\alpha(g_{n-2},g_{n-1})\alpha(g^{-1}_{n-1},g_{n-1})\alpha(1,g_n)\alpha(g_{n-1},g_{n-1}^{-1})^{-1}\alpha(g_{n-1}g_n,g_{n}^{-1}g_{n-1}^{-1})^{-1},\\
&\varphi(\overline{g_1}\otimes\ldots \otimes \overline{g_{n-2}g_{n-1}}\cdot \overline{g_n}\otimes\overline{g_n}^{^{-1}}\otimes\overline{g_n}\cdot \overline{g_{n}^{-1}g_{n-1}^{-1}})\cdot \overline{g_{n-1}g_n}\\
&=-\alpha(g_{n-2},g_{n-1})\alpha(g_{n-2}g_{n-1},g_n)\alpha(g_n,g_{n}^{-1}g_{n-1}^{-1})\alpha(g_n,g_n^{-1})^{-1}\alpha(g_{n-1}g_n,g_n^{-1}g_{n-1}^{-1})^{-1},\\
&\varphi(\overline{g_1}\otimes\ldots\otimes \overline{g_{n-2}g_{n-1}g_n}\otimes\overline{g_n^{-1}}\otimes \overline{g_{n-1}^{-1}})\cdot \overline{g_{n-1}g_n}\\
&=-\alpha(g_{n-1},g_n)\alpha(g_{n-2},g_{n-1}g_n)\alpha(g_n,g_n^{-1})^{-1}\alpha(g_{n-1}g_n,g_n^{-1}g_{n-1}^{-1})^{-1}
\alpha(g_n,g_n^{-1}g_{n-1}^{-1}),\\
&\varphi(\overline{g_1}\otimes\ldots \otimes \overline{g_{n-2}g_{n-1}g_n}\otimes \overline{g_n^{-1}}\otimes \overline{g_{n-1}^{-1}})\cdot \overline{g_{n-1}g_n}\\
&=-A\alpha(g_n,g_n^{-1}g_{n-1}^{-1}) \alpha(g_n^{-1},g_{n-1}^{-1})\alpha(g_n^{-1},g_{n-1}^{-1})^{-1}\varphi(\overline{g_1}\otimes\ldots \otimes \overline{g_{n-2}g_{n-1}g_n}\otimes \overline{g_n^{-1}}\otimes \overline{g_{n-1}^{-1}})\cdot \overline{g_{n-1}g_n}\\
&=-\alpha(g_{n-1},g_n)\alpha(g_{n-2},g_{n-1}g_n)\alpha(g_{n-1}g_n,g_n^{-1}g_{n-1}^{-1})^{-1}\alpha(g_n^{-1},g_{n-1}^{-1})^{-1}\\
&\quad\times\varphi(\overline{g_1}\otimes\ldots \otimes \overline{g_{n-2}g_{n-1}g_n}\otimes \overline{g_n^{-1}}\otimes \overline{g_{n-1}^{-1}})\cdot \overline{g_{n-1}g_n};
\end{align*}
where, in the sixth equality we used the notation $$A=\alpha(g_{n-1},g_n)\alpha(g_{n-2},g_{n-1}g_n)\alpha(g_n,g_n^{-1})^{-1}\alpha(g_{n-1}g_n,g_n^{-1}g_{n-1}^{-1})^{-1}.$$
On the other hand we have
\begin{align*}
&\sigma_n(\sigma_{n-1}(\sigma_n\varphi))(\overline{g_1}\otimes \ldots \otimes \overline{g_n})=-\sigma_{n-1}(\sigma_n\varphi)(\overline{g_1}\otimes\ldots\otimes \overline{g_{n-1}}\cdot \overline{g_n}\otimes \overline{g_n}^{^{-1}})\cdot \overline{g_n}\\
&=-\alpha(g_{n-1},g_n)\alpha(g_{n-2},g_{n-1}g_n)\alpha(g_{n-1}g_n,g_n^{-1}g_{n-1}^{-1})^{-1}\alpha(g_n,g_n^{-1})^{-1}
\alpha(g_{n-1}g_n,g_n^{-1})\alpha(g_n^{-1}g_{n-1}^{-1},g_{n-1}),\\
&\alpha(g_{n-1},g_{n-1}^{-1})^{-1}\alpha(g_{n-1},g_n)\varphi(\overline{g_1}\otimes\ldots \otimes \overline{g_{n-2}g_{n-1}g_n}\otimes \overline{g_n^{-1}}\otimes \overline{g_{n-1}^{-1}})\cdot \overline{g_{n-1}g_n}\\
&=-\alpha(g_{n-1},g_n)\alpha(g_{n-2},g_{n-1}g_n)\alpha(g_{n-1}g_n,g_n^{-1}g_{n-1}^{-1})^{-1}\alpha(g_n^{-1},g_{n-1}^{-1})^{-1}\\
&\quad \times\varphi(\overline{g_1}\otimes\ldots \otimes \overline{g_{n-2}g_{n-1}g_n}\otimes \overline{g_n^{-1}}\otimes \overline{g_{n-1}^{-1}})\cdot \overline{g_{n-1}g_n}.
\end{align*}
The relations $\sigma_i(\sigma_j\varphi)=\sigma_j(\sigma_i\varphi)$ for $\mid i-j\mid \geq 2$ are easily checked.
In most of the above equalities we applied heavily the identities (\ref{eq31}), (\ref{eq32}) and (\ref{eq33}).

We also have to verify that the above action is compatible with $\delta^*,$ that is
\[\delta^{n}(\varphi)\in \left(C^{n+1}(k_{\alpha}G,M)\right)^{\Sigma_{n+2}} \mbox{ if } \varphi\in \left(C^{n}(k_{\alpha}G,M)\right)^{\Sigma_{n+1}}.\]
Recall that $\delta^{n}=\sum_{j=0}^{n+1}(-1)^jd_j.$ The same formulas as in \cite[Proposition 5.1]{St1} hold:
\begin{align*}
&\sigma_i(d_j(\varphi))=\begin{cases}d_j(\sigma_i\varphi), \mbox{ if } i<j\\
d_j(\sigma_{i-1}\varphi), \mbox{ if } i\geq j+2\end{cases}\\
& \mbox{  If } i=j \mbox{  we have } (i,i+1)d_i(\varphi)=\sigma_i(d_i(\varphi))=-d_{i-1}(\varphi)\\
& \mbox{  If } i=j+1 \mbox{ we have } \sigma_i(d_{i-1}(\varphi))=-d_i(\varphi).
\end{align*}
The only formal differences with respect to symmetric group cohomology case are  for $i=n+1$ (see the last formula of \ref{action_on_chain}) or $j=n+1$ (see $d_{n+1}$ in Remark \ref{rem}). We exemplify with some of these situations.

Set $i=n+1$ and $j=n+1,$ then
\begin{align*}
&\sigma_{n+1}(d_{n+1}(\varphi))(\overline{g_1}\otimes \ldots \otimes \overline{g_{n+1}})\\
&=-d_{n+1}(\varphi)(\overline{g_1}\otimes \ldots \otimes \overline{g_n}\cdot \overline{g_{n+1}}\otimes \overline{g_{n+1}}^{^{-1}})\cdot \overline{g_{n+1}}\\
&=-\varphi(\overline{g_1}\otimes \ldots \otimes \overline{g_n}\cdot \overline{g_{n+1}}) \overline{g_{n+1}}^{^{-1}}\cdot \overline{g_{n+1}}\\
&=-d_n(\varphi)(\overline{g_1}\otimes\ldots\otimes\overline{g_{n+1}}).
\end{align*}

Let $i=n+1$ and $j=n,$ then
\begin{align*}
&\sigma_{n+1}(d_{n}(\varphi))(\overline{g_1}\otimes \ldots \otimes \overline{g_{n+1}})\\
&=-d_{n}(\varphi)(\overline{g_1}\otimes \ldots \otimes \overline{g_n}\cdot \overline{g_{n+1}}\otimes \overline{g_{n+1}}^{^{-1}})\cdot \overline{g_{n+1}}\\
&=-\varphi(\overline{g_1}\otimes \ldots \otimes \overline{g_{n-1}}\otimes \overline{g_n})\cdot \overline{g_{n+1}}\\
&=-d_{n+1}(\varphi)(\overline{g_1}\otimes \ldots \otimes \overline{g_{n-1}}\otimes \overline{g_n}\otimes \overline{g_{n+1}}).
\end{align*}
Finally, if we choose $i=n$ and $j=n+1,$ then
\begin{align*}
&\sigma_n(d_{n+1}(\varphi))(\overline{g_1}\otimes \ldots \otimes \overline{g_{n}}\otimes \overline{g_{n+1}})\\
&=-d_{n+1}(\varphi)(\overline{g_1}\otimes \ldots \otimes \overline{g_{n-1}}\cdot \overline{g_n}\otimes \overline{g_n}^{^{-1}}\otimes \overline{g_n}\cdot \overline{g_{n+1}})\\
&=-\varphi(\overline{g_1}\otimes \ldots \otimes \overline{g_{n-1}}\cdot \overline{g_n}\otimes \overline{g_n}^{^{-1}})\cdot\overline{g_n}\cdot \overline{g_{n+1}}\\
&=(\sigma_n \varphi)(\overline{g_1}\otimes \ldots \otimes \overline{g_n})\cdot \overline{g_{n+1}}\\
&=d_{n+1}(\sigma_n\varphi)(\overline{g_1}\otimes \ldots \otimes \overline{g_{n}}\otimes \overline{g_{n+1}}).
\end{align*}
\end{proof}

\begin{definition}\label{defn33}
The symmetric Hochschild cohomology of $k_{\alpha}G$ with coefficients in $M,$ denoted $$\HHS^n(k_{\alpha}G,M),$$ is the homology of $CS^n(k_{\alpha}G,M)=\left(C^n(k_{\alpha}G, M)\right)^{\Sigma_{n+1}},$ the subcomplex of $\Sigma_{n+1}$-invariants of $C^n(k_{\alpha}G, M).$
\end{definition}
\section{Additive centralizers decomposition of Hochschild cohomology of twisted group algebras}\label{sec4}
We continue with the notations from Introduction. We define the $k$-linear maps between cohomology spaces which allow to describe (\ref{eq11}) and  (\ref{eq12}) explicitly in terms of cochain maps.
This approach is similar to that of \cite[Remark 5.2, b)]{LiZh} or \cite[Theorem 6.3]{LiZh}, presented for group algebras.

\begin{subsec}\label{prima'}
 Let $x\in X.$ We define the $k$-linear maps $\pi^*_{G,x,\alpha}$ between cochain complexes as
 \begin{align*}&\pi^n_{G,x,\alpha}:C^n(k_{\alpha}G,k_{\alpha}G)\to C^n(C_G(x),k\bar{x})\\
   &(\varphi:(k_{\alpha}G)^{\otimes n}\to k_{\alpha}G)\mapsto (\pi^n_{G,x,\alpha}(\varphi):(C_G(x))^{\times n}\to k\bar{x})\\
   & \pi^n_{G,x,\alpha}(\varphi)(h_1,h_2,\ldots, h_n)=a_{1,x}\bar{x}, \mbox{ for all } h_1, h_2, \ldots, h_n\in C_G(x),
 \end{align*}
 where $a_{1,x}$ is the coefficient of $\bar{x}$ in $\varphi(\overline{h_1}\otimes \overline{h_2}\otimes \ldots \otimes \overline{h_n})\cdot \overline{h_n}^{^{-1}}\cdot \overline{h_{n-1}}^{^{-1}}\cdot \ldots \cdot \overline{h_1}^{^{-1}}.$
\end{subsec}

\begin{subsec}\label{prima''}
 Let $x\in Z(G).$ We define the $k$-linear maps $\nu_{G,x,\alpha}^*$ between cochain complexes as
 \begin{align*}
 &\nu^n_{G,x,\alpha}:C^n(G,k\bar{x})\to C^n(k_{\alpha}G,k_{\alpha}G)\\
 &(\psi:G^{\times n}\to k\bar{x})\mapsto (\nu^n_{G,x,\alpha}(\psi):\mathcal{B}^{\otimes n}\to k_{\alpha}G)\\
 &\nu^n_{G,x,\alpha}(\psi)(\overline{g_1}\otimes \overline{g_2}\otimes \ldots \otimes \overline{g_n})=\psi(g_1,\ldots, g_n)\cdot \overline{g_1}\cdot \ldots \cdot \overline{g_n}, \mbox{  for all }  g_1, \ldots, g_n\in G
 \end{align*}
\end{subsec}

\begin{subsec}\label{trivial_def} \label{5.3}
 Let $x\in X$ and $\alpha=\one_k$. In this case we identify $\mathcal{B}$ with $G$ and renounce  the ''overline'' notation. Recall that $\nu_{G,x,\alpha}^*$ is denoted $\nu_{G,x}^*$ and, we define
\[\nu^n_{G,x}:C^n(G,kx)\to C^n(kG,kG) \mbox{ by }\]
\[\nu^n_{G,x}(\psi)(g_1\otimes \ldots \otimes g_n)=\sum_{j=1}^{n_x}\psi(h_{j,1},\ldots, h_{j,n})x^{-1}x_jg_1\ldots g_n\]
for all $\psi\in C^n(C_G(x),kx)$ and $g_1,\ldots, g_n\in G.$ Here $h_{j,1}, \ldots h_{j,n}\in C_G(x)$ are determined by the sequence $\{g_1,\ldots, g_n\}$ and the elements $x_j,$ with $j\in \{1,\ldots, n_x\},$ are given as we explain below.

Let $C_x=\{gxg^{-1}\mid g\in G\}$. We choose a right coset decomposition of $C_G(x)$ in $G$
\[G=C_G(x)\gamma_{1,x}\cup\ldots \cup C_G(x)\gamma_{n_x,x}=\gamma_{1,x}^{-1}C_G(x)\cup\ldots\cup\gamma_{n_x,x}^{-1}C_G(x)
\] such that

\[C_x=\{x, \gamma_{2,x}^{-1}\ x\ \gamma_{2,x},\ldots, \gamma_{n_x,x}^{-1}\ x\ \gamma_{n_x,x}\}.\]

Then we set $x_j:=\gamma_{j,x}^{-1}\ x\ \gamma_{j,x}$ for all $j\in \{1,\ldots n_{x}\}$ and $\gamma_{1,x}=1.$
Let $j\in \{1,\ldots, n_x\}.$ Since $\gamma_{j,x}g_1\in G$ it follows that there is a unique $s_j^1\in \{1,\ldots, n_x\}$ such that
\[\gamma_{j,x}g_1=h_{j,1}\gamma_{s_j^1,x},\]
with $h_{j,1}\in C_G(x).$ By the same arguments there are indexes (uniquely determined) $s_j^2,\ldots, s_j^n\in \{1,\ldots, n_x\}$ such that $h_{j,2},\ldots h_{j,n}\in C_G(x)$ are determined by the sequence $g_2,\ldots, g_n$ as follows

\[\gamma_{s_j^1,x}g_2=h_{j,2}\gamma_{s_j^2,x},\ \ldots, \ \gamma_{s_j^{n-1},x\ }g_n=h_{j,n}\gamma_{s_j^n,x}.\]

\end{subsec}
\begin{remark}
For a twisted group algebra,  if we attempt a similar definition to that given in \ref{trivial_def}, it will not be clear how to verify the commutativity of $\nu_{G,x,\alpha}^*$ with the differential. If $x\in Z(G)$
and $\alpha=\one_k$ then \ref{trivial_def} agrees with \ref{prima''}.
\end{remark}
\begin{lemma}\label{lemma4.5} Let $n\in\Z,n\geq 1$.The following hold:
\begin{itemize}
\item[$a)$] Let $x\in X.$ Then $\pi^*_{G,x,\alpha}$ commutes with the differentials. That is, the next diagram is commutative
    \[\xymatrix{ C^{n-1}(k_{\alpha}G, k_{\alpha}G)\ar[d]^{\pi^{n-1}_{G,x,\alpha}}\ar[r]^{\delta^{n-1}} &C^n(k_{\alpha}G,k_{\alpha}G)\ar[d]^{\pi^{n}_{G,x,\alpha}} \\
   C^{n-1}(C_G(x),k\bar{x})\ar[r]^{\partial^{n-1}}&C^{n}(C_G(x),k\bar{x})}\]

\item[$b)$] Let $x\in Z(G).$ Then $\nu^*_{G,x,\alpha}$ commutes with the differentials. That is, the next diagram is commutative
   \[ \xymatrix{  C^{n-1}(k_{\alpha}G, k_{\alpha}G)\ar[r]^{\delta^{n-1}} &C^n(k_{\alpha}G,k_{\alpha}G) \\
   C^{n-1}(G,k\bar{x})\ar[u]_{\nu^{n-1}_{G,x,\alpha}}\ar[r]^{\partial^{n-1}}&C^{n}(G,k\bar{x})\ar[u]_{\nu^{n}_{G,x,\alpha}}
    }\]

\item[$c)$]We have:\[ \left(\bigoplus_{x\in Z(G)}\pi^n_{G,x,\alpha}\right)\circ \left(\bigoplus_{x\in Z(G)}\nu^n_{G,x,\alpha}\right)=\bigoplus_{x\in Z(G)}\id_{C^n(G,k\bar{x})} \mbox{   }\]

\item[$d)$] Consider $x\in X$ and $\alpha=\one_k$. The following diagram is commutative
     \[ \xymatrix{  C^{n-1}(k_{}G, k_{}G)\ar[r]^{\delta^{n-1}} &C^n(k_{}G,k_{}G) \\
   C^{n-1}(C_G(x),kx)\ar[u]_{\nu^{n-1}_{G,x}}\ar[r]^{\partial^{n-1}}&C^{n}(C_G(x),kx)\ar[u]_{\nu^{n}_{G,x}}
    }\]
\item[$e)$] And
\[ \left(\bigoplus_{x\in X}\pi^n_{G,x}\right)\circ \left(\bigoplus_{x\in X}\nu^n_{G,x}\right)=\bigoplus_{x\in X}\id_{C^n(C_G(x),kx)},\] where $\pi^n_{G,x}$ stands for $\pi^n_{G,x,\one_k}.$
\end{itemize}
\end{lemma}
\begin{proof}
$a)$ Let $h_1, \ldots, h_n\in C_G(x)$ and the $k$-linear map $\varphi:(k_{\alpha}G)^{\otimes(n-1)}\rightarrow k_{\alpha}G$. According to \ref{prima'} and to the definitions of $\delta^{n-1}$ and $\partial^{n-1}$ we obtain
\[\pi^n_{G,x,\alpha}(\delta^{n-1}(\varphi))(h_1,\ldots, h_n)=a_{1,x}\bar{x},\] where $a_{1,x}$ is the coefficient of $\bar{x}$ in \[ \delta^{n-1}(\varphi)(\overline{h_1}\otimes\ldots\otimes \overline{h_n})\cdot \overline{h_n}^{^{-1}}\cdot \ldots \cdot \overline{h_1}^{^{-1}} .\]
Consequently, $a_{1,x}$ is the coefficient of  $\bar{x}$ in

\begin{align*}&( \overline{h_1}\cdot \varphi(\overline{h_2}\otimes\ldots\otimes \overline{h_n})+\sum_{i=1}^{n-1}(-1)^i\varphi(\overline{h_1}\otimes\ldots\otimes\overline{h_{i-1}}\cdot \overline{h_i}\otimes \ldots \otimes \overline{h_n})\\
&\left.+(-1)^n\varphi(\overline{h_1}\otimes\ldots\otimes \overline{h_{n-1}})   \cdot \overline{h_n} \ \right) \cdot \overline{h_n}^{^{-1}}\cdot \ldots  \cdot \overline{h_1}^{^{-1}}.
\end{align*}
Hence $a_{1,x}=b^0_{1,x}+\sum_{i=1}^{n-1}(-1)^ib_{1,x}^i+(-1)^nb_{1,x}^n,$
where $b^0_{1,x}$ is the coefficient of $\bar{x}$ in $$\overline{h_1}\cdot \varphi(\overline{h_2}\otimes\ldots\otimes \overline{h_n})\cdot \overline{h_n}^{^{-1}}\cdot \ldots \cdot \overline{h_1}^{^{-1}}.$$ Further, for $i\in \{1,\ldots, n-1\},$ $b_{1,x}^i$ is the coefficient of $\bar{x}$ in
\begin{align*}
&\alpha(h_{i-1},h_i)\alpha(h_{i-1},h_{i-1}^{-1})^{-1}\alpha(h_i,h_i^{-1})^{-1}\alpha(h_{i}^{-1},h_{i-1}^{-1})\\
&\varphi(\overline{h_1}\otimes\ldots\otimes\overline{h_{i-1}h_i}\otimes\ldots\otimes \overline{h_n})
\cdot \overline{h_n}^{^{-1}}\cdot \ldots\cdot\overline{h_i^{-1}h_{i-1}^{-1}}\cdot \ldots\cdot \overline{h_1}^{^{-1}}
\end{align*}
and $b_{1,x}^n$ is the coefficient of $\bar{x}$ in $\varphi(\overline{h_1}\otimes\ldots \otimes \overline{h_{n-1}})\cdot\overline{h_{n-1}}^{^{-1}}\cdot \ldots \cdot \overline{h_{n-1}}^{^{-1}}.$

On the other hand we have
\begin{align*}
&\partial^{n-1}(\pi^{n-1}_{G,x,\alpha}(\varphi))(h_1,\ldots, h_n)=h_1\pi^{n-1}_{G,x,\alpha}(\varphi)(h_2,\ldots, h_n)\\
&+\sum_{i=1}^{n-1}(-1)^i\pi^{n-1}_{G,x,\alpha}(\varphi)(h_1,\ldots, h_{i-1}h_i,\ldots, h_n)+(-1)^n\pi^{n-1}_{G,x,\alpha}(\varphi)(h_1,\ldots, h_{n-1})\\
&=h_1(a^0_{1,x}\bar{x})+\sum_{i=1}^{n-1}(-1)^ia^i_{1,x}\bar{x}+(-1)^na_{1,x}^n\bar{x}.
\end{align*}
Here $a_{1,x}^0$ is the coefficient of $\bar{x}$  in \[\varphi(\overline{h_2}\otimes\ldots\otimes\overline{h_n})\cdot \overline{h_n}^{^{-1}}\cdot \ldots\cdot \overline{h_2}^{^{-1}},\] $a_{1,x}^i $ is the coefficient of $\bar{x}$ in
\[\alpha(h_{i-1}h_i,h^{-1}_ih_{i-1}^{-1})^{-1}\varphi(\overline{h_1}\otimes\ldots\otimes\overline{h_{i-1}h_i}\otimes
\ldots\otimes\overline{h_n})\cdot \overline{h_n}^{^{-1}}\cdot\ldots\cdot \overline{h_i^{-1}h_{i-1}^{-1}}\cdot \ldots\cdots \overline{h_1}^{^{-1}}\]
and $a_{1,x}^n$ is the coefficient of $\bar{x}$ in \[\varphi(\overline{h_1}\otimes\ldots\otimes\overline{h_{n-1}})\cdot\overline{h_{n-1}}^{^{-1}}\cdot \ldots \cdot \overline{h_1}^{^{-1}}.\] It follows that $a_{1,x}^n=b_{1,x}^n.$

 The following equality holds
 \[h_1(a_{1,x}^0\bar{x})=\alpha(h_1,x)\alpha(x,h_1)^{-1}a_{1,x}^0\bar{x}=b_{1,x}^0\bar{x},\] since $b^0_{1,x}$ is the coefficient of $\bar{x}$ in \[\overline{h_1}\cdot (\varphi(\overline{h_2}\otimes\ldots\otimes \overline{h_n})\cdot \overline{h_n}^{^{-1}}\cdot \ldots \cdot \overline{h_2}^{^{-1}})\cdot \overline{h_1}^{^{-1}}.\]
 It is easily checked that $a_{1,x}^i=b_{1,x}^i$ for all $i\in \{1,\ldots, n-1\}$ since $\alpha(1,h_{i-1}^{-1})=1.$
This concludes the proof of the first statement.

$b)$ Consider the map $\psi:G^{\times n}\to k\bar{x}$ and the elements $g_1,\ldots, g_n\in G.$ According to \ref{prima''} we have

\begin{align*}
&\delta^{n-1}(\nu_{G,x,\alpha}^{n-1}(\psi))(\overline{g_1}\otimes\ldots\otimes\overline{g_n})=\overline{g_1}\cdot \psi(g_2,\ldots, g_n)\cdot \overline{g_2}\cdot \ldots\cdot \overline{g_n}\\
&\quad+\sum_{i=1}^{n-1}(-1)^i\alpha(g_{i-1},g_i)\psi(g_1,\ldots, g_{i-1}g_i,\ldots, g_n)\cdot \overline{g_1}\cdot \ldots \cdot \overline{g_{i-1}g_i}\cdot \ldots\cdot \overline{g_n}\\
&\quad +(-1)^n\psi(g_1,\ldots, g_n)\cdot \overline{g_1}\cdot \ldots\cdot \overline{g_{n-1}}\cdot \overline{g_n}.\\
\end{align*}
On the other hand we get
\begin{align*}
&\nu^n_{G,x,\alpha}(\partial^{n-1}(\psi))(\overline{g_1}\otimes \ldots \otimes \overline{g_n})=\partial^{n-1}(\psi)(g_1,\ldots, g_n)\cdot \overline{g_1}\cdot \ldots \cdot \overline{g_n}\\
&=\left(g_1\psi(g_2,\ldots,g_n)+\sum_{i=1}^{n-1}(-1)^i\psi(g_1,\ldots , g_{i-1}g_i,\ldots, g_n)+(-1)^n\psi(g_1,\ldots, g_{n-1})\right)\cdot \overline{g_1}\cdot \ldots \cdot \overline{g_n}\\
&=\overline{g_1}\cdot \psi(g_2,\ldots, g_n)\cdot \overline{g_2}\cdot \ldots\cdot \overline{g_n}\\
&\quad +\sum_{i=1}^{n-1}(-1)^i\alpha(g_{i-1},g_i)\psi(g_1,\ldots, g_{i-1}g_i,\ldots, g_n)\cdot \overline{g_1}\cdot \ldots \cdot \overline{g_{i-1}g_i}\cdot \ldots\cdot \overline{g_n}\\
&\quad +(-1)^n\psi(g_1,\ldots, g_n)\cdot \overline{g_1}\cdot \ldots\cdot \overline{g_{n-1}}\cdot \overline{g_n}.\\
\end{align*}

$c)$ Let $\psi:G^{\times n}\to k\bar{x},$ $x\in Z(G)$ and $g_1,\ldots, g_n\in G.$ Notice that
\[\pi_{G,x,\alpha}^n(\nu_{G,x,\alpha}^n(\psi))(g_1,\ldots, g_n)=a_{1,x}\bar{x},\] where $a_{1,x}$ is the coefficient of $\bar{x}$ in
\begin{align*}
&\nu_{G,x,\alpha}^n(\psi)(\overline{g_1}\otimes\ldots\otimes\overline{g_n})\cdot \overline{g_n}^{^{-1}}\cdot \ldots \cdot \overline{g_1}^{^{-1}}\\
&=\psi(g_1,\ldots, g_n)\cdot \overline{g_1}\cdot \ldots \cdot \overline{g_n}\cdot \overline{g_n}^{^{-1}}\cdot \ldots \cdot \overline{g_1}^{^{-1}}\\
&=\psi(g_1,\ldots, g_n).
\end{align*}
We obtain $\pi^n_{G,x,\alpha}\circ \nu^n_{G,x,\alpha}=\id_{C^n(G,k\bar{x})}.$

Statement $d)$ is a consequence of \cite[Theorem 6.3]{LiZh}, since the definition of  $\nu_{G,x}^*$ is similar to the general one.

$e)$  We will show that $\pi_{G,x}^n\circ \nu_{G,x}^n=\id_{C^n(C_G(x),kx)}$ for some $x\in X$. Let $\psi:(C_G(x))^{\times n}\rightarrow kx$ be a set map
and $h_1,\ldots, h_n\in C_G(x)$.
By \ref{prima'} we obtain that
\[\pi_{G,x}^n(\nu_{G,x}^n(\psi))(h_1,\ldots,h_n)=a_{1,x}x,\] where $a_{1,x}$ is the of coefficient of $x$ in
\[\nu_{G,x}^n(\psi)(h_1\otimes\ldots\otimes h_n)h_n^{-1}\ldots h_1^{-1}=\sum_{j=1}^{n_x}\psi(h_{j,1},\ldots,h_{j,n})x^{-1} x_j \]
The last above equality follows from \ref{trivial_def}. If $j=1$ then $x_1=x$ and $\gamma_{1,x}=1$ thus
\[\gamma_{1,x}h_1=h_1\gamma_{s_1^1,x}\]
with $\gamma_{s_1^1,x}=1$ so $h_{1,1}=h_1$. We continue with the equations from \ref{trivial_def} to obtain
$h_{1,2}=h_2,\ldots, h_{1,n}=h_n$. Since if $j\in \{2,\ldots,n_x\}$ then $x_j\neq x$, we obtain that $a_{1,x}$ is the coefficient of $x$ in
$\psi(h_1,\ldots, h_n)$, which is what we need.
\end{proof}

Lemma \ref{lemma4.5} above is useful for the following proposition.
\begin{proposition}\label{proposition4.6} Let $n\in \Z,n\geq 0$.
\begin{itemize}
\item[$a)$] Let $x\in X.$ Then $\pi^*_{G,x,\alpha}$ induces the well-defined maps in cohomology
\[\pi^{*}_{G,x,\alpha}:\HH^*(k_{\alpha}G) \to \Hrm^*(C_G(x),k\bar{x}).\]
\item[$b)$] Let $x\in Z(G).$ Then $\nu^*_{G,x,\alpha}$ induces the well-defined maps in cohomology
\[\nu^{*}_{G,x,\alpha}:\Hrm^*(G,k\bar{x})\to \HH^*(k_{\alpha}G).\]
\item[$c)$] There is an embedding of cohomology spaces
\[\bigoplus_{x\in Z(G)}\nu^{n}_{G,x,\alpha}:\bigoplus_{x\in Z(G)}\Hrm^n(G,k\bar{x})\rightarrowtail \HH^n(k_{\alpha}G)\]
with
\[\bigoplus_{x\in Z(G)}\pi^{n}_{G,x,\alpha}:\HH^n(k_{\alpha}G)\twoheadrightarrow \bigoplus_{x\in Z(G)}\Hrm^n(G,k\bar{x}) \] its left inverse.
 \item[$d)$] If $G$ is abelian then there is an additive  centralizers group cohomology decomposition
 \[\bigoplus_{x\in G}\nu^{n}_{G,x,\alpha}:\bigoplus_{x\in G}\Hrm^n(G,k\bar{x})\simeq \HH^n(k_{\alpha}G).\]

\end{itemize}
\end{proposition}
\begin{proof}
 The first three statements follow immediately from statements $a),$ $b)$ and $c)$ of Lemma \ref{lemma4.5}.
 From statement $c)$ of the same lemma, we know that $\bigoplus_{x\in G}\nu^n_{G,x,\alpha}$ is a $k$-linear monomorphism.
 Since the isomorphisms given in (\ref{eq11}) determine the equality
 \[\mathrm{dim}_k(\HH^n(k_{\alpha}G))=\bigoplus_{x\in G}\mathrm{dim}_k(\Hrm^n(G,k\bar{x})),\] the result of $d)$ is obvious since $G$ is abelian.
 \end{proof}

 \begin{remark}
 \begin{itemize}

\item[a)] For the trivial 2-cocyle $\alpha=\one_k,$ by Lemma \ref{lemma4.5} statement $e)$, we obtain the splitting $k$-monomorphism
\[\nu^n_{G,x}:\Hrm^n(C_G(x),k)\to \HH^n(kG)\] for any $x\in X.$ We identified here $kx$ to $k$ as trivial $kC_G(x)$-modules.
  Consequently, we recover the explicit additive centralizers decomposition from \cite[Remark 5.2 or Theorem 6.3]{LiZh}
  \[\HH^n(kG)\simeq \bigoplus_{x\in X}\Hrm^n(C_G(x),k)\] by applying the same arguments as in Proposition \ref{proposition4.6}, statement $d).$
  \item[b)] In \cite[Theorem 2.1]{CiSo} Cibils and Solotar obtained an isomorphism of graded rings between Hochschild cohomology ring of a group algebra
  $\HH^*(kG)$ and $kG\otimes \Hrm^*(G,k)$, if $G$ is abelian. Explicit formulas were further developed, for non-commutative groups also, in \cite{SiWi} and more recently
  in \cite{LiZh}. Proposition \ref{proposition4.6}, statement d) can be viewed as an additive version of this result for twisted group algebras. We ask if someone can develop a graded ring isomorphism explicit formula in this case?
  \end{itemize}
 \end{remark}
\section{An embedding of additive centralizers decomposition of symmetric cohomology into symmetric Hochschild cohomology}\label{sec5}

The aim of this section is to prove that $\pi^*_{G,x,\alpha},$ $\nu^*_{G,x,\alpha}$ and $\nu^*_{G,x}$ are $\Sigma_{*+1}$-maps.

\begin{proposition}\label{proposition5.1}
Let $n\in \Z,$ $n\geq 0$ and consider the twisted group algebra $k_{\alpha}G.$ The following statements hold.
\begin{itemize}
\item[$a)$] Let $x\in X$ and $\varphi\in C^n(k_{\alpha}G,k_{\alpha}G).$ Then
\[\pi^n_{G,x,\alpha}(\tau \varphi)=\tau\pi^n_{G,x,\alpha}( \varphi)\]
for any $\tau\in \Sigma_{n+1};$

\item[$b)$] Let $x\in Z(G)$ and $\psi\in C^n(G,k\bar{x}).$ Then
\[\nu^n_{G,x,\alpha}(\tau \psi)=\tau\nu^n_{G,x,\alpha}( \psi)\]
for any $\tau\in \Sigma_{n+1};$

\item[$c)$] Let $x\in X$ and $\psi\in C^n(C_G(x),kx).$ Then
\[\nu^n_{G,x}(\tau \psi)=\tau\nu^n_{G,x}( \psi)\]
for any $\tau\in \Sigma_{n+1}.$
\end{itemize}

\begin{proof}
It suffices to verify the statements for the generators $(i,i+1)\in \Sigma_{n+1},$ where $i\in \{1,\ldots, n\}.$

$a)$ Let $h_1,\ldots, h_n\in C_G(x).$ For $i=1$ we have
\[\pi^n_{G,x,\alpha}((1,2)\varphi)=a_{1,x}\bar{x},\] where $a_{1,x}$ is the coefficient of $\bar{x}$ in
\[-\overline{h_1}\cdot(\varphi(\overline{h_1}^{^{-1}}\otimes\overline{h_1}\cdot\overline{h_2}\otimes \ldots\otimes\overline{h_n})\cdot \overline{h_n}^{^{-1}}\cdot\ldots \cdot \overline{h_2}^{^{-1}})\cdot \overline{h_1}^{^{-1}}.\]

 If $a'$ is the coefficient of $\bar{x}$ in
 $$A:=\varphi(\overline{h_1}^{^{-1}}\otimes \overline{h_1}\cdot \overline{h_2}\otimes\ldots\otimes\overline{h_n})\cdot \overline{h_n}^{^{-1}}\cdot \ldots \overline{h_2}^{^{-1}}$$ then $$a_{1,x}=-a'\alpha(h_1,x)\alpha(x,h_1)^{-1}$$ and
 \begin{align*}
 ((1,2)\pi^n_{G,x,\alpha}(\varphi))(h_1,\ldots, h_n)&=-h_1\pi^n_{G,x,\alpha}(\varphi)(h_1^{-1}\otimes h_1h_2\otimes\ldots\otimes h_n)\\
 &=-h_1(a_{1,x}'\bar{x})\\
 &=-a_{1,x}'\alpha(h_1,x)\alpha(x,h_1)^{-1}\bar{x}.
 \end{align*}
 Here $a'_{1,x}$ is the coefficient of $\bar{x}$ in
 \begin{align*}
&\varphi(\overline{h_1^{-1}}\otimes\overline{h_1h_2}\otimes\ldots \otimes \overline{h_n})\cdot \overline{h_n}^{^{-1}}\cdot\ldots \cdot \overline{h_1h_2}^{^{-1}}\cdot \overline{h_1^{-1}}^{^{-1}}\\
&=\alpha(h_1,h_1^{-1})\alpha(h_1,h_2)^{-1}\alpha(h_1^{-1},h_1)^{-1}\alpha(h_1h_2,h_2^{-1}h_1^{-1})^{-1}\alpha(h_2,h_2^{-1})
\alpha(h_2^{-1}h_1^{-1},h_1)A\\
&=(\alpha(h_1,h_1^{-1})\alpha(h_2,h_2^{-1}h_1^{-1}))^{-1}\alpha(h_2^{-1}h_1^{-1},h_1)\alpha(h_2,h_2^{-1})A\\
&=(\alpha(h_1,h_1^{-1})\alpha(h_2,h_2^{-1}h_1^{-1})\alpha(h_2^{-1},h_1^{-1}))^{-1}\alpha(h_2^{-1},h_1^{-1})
\alpha(h_2^{-1}h_1^{-1},h_1)\alpha(h_2,h_2^{-1})A\\
&=A,
 \end{align*}
 hence $a_{1,x}'=a'.$

 For $1<i<n$ we similarly verify that
 \[\pi^n_{G,x,\alpha}((i,i+1)\varphi)=(i,i+1)\pi^n_{G,x,\alpha}(\varphi).\]
 Set $i=n.$ We have
 \[\pi^n_{G,x,\alpha}((n,n+1)\varphi)(h_1,\ldots,h_n)=a_{1,x}\bar{x},\] where $a_{1,x}''$ is the coefficient of $\bar{x}$ in
 \begin{align*}
 &((n,n+1)\varphi)(\overline{h_1}\otimes\ldots\otimes\overline{h_n})\cdot \overline{h_n}^{^{-1}}\cdot \ldots \cdot \overline{h_1}^{^{-1}}\\
 &=-\varphi(\overline{h_1}\otimes\ldots\otimes \overline{h_{n-1}h_n}\otimes\overline{h_n^{-1}})\cdot \overline{h_{n-1}}^{^{-1}}\cdot \ldots \cdot \overline{h_1}^{^{-1}}\alpha(h_{n-1},h_n)\alpha(h_n,h_n^{-1})^{-1}.
 \end{align*}
Further we get
\[((n,n+1)\pi^n_{G,x,\alpha}(\varphi))(h_1,\ldots, h_n)=-\pi^n_{G,x,\alpha}(h_1,\ldots, h_{n-1}h_n, h_n^{-1})=a'_{1,x}\bar{x},\]
where $a_{1,x}'''$ is the coefficient of $\bar{x}$ in
\begin{align*}
&-\varphi(\overline{h_1}\otimes\ldots\otimes\overline{h_{n-1}h_n}\otimes\overline{h_n^{-1}})\cdot \overline{h_n^{-1}}^{^{-1}}\cdot \overline{h_{n-1}h_n}^{^{-1}}\cdot \ldots \cdot \overline{h_1}^{^{-1}}\\
&=-\varphi(\overline{h_1}\otimes\ldots\otimes\overline{h_{n-1}h_n}\otimes\overline{h_n^{-1}})\cdot \overline{h_n}\cdot \overline{h_n}^{^{-1}}\cdot \overline{h_{n-1}}^{^{-1}}\cdot \ldots \cdot \overline{h_1}^{^{-1}}\alpha(h_n^{-1},h_n)^{-1}\alpha(h_{n-1},h_n).
\end{align*}
Thus $a_{1,x}''=a_{1,x}'''.$

$b)$ The cases $i=1$ and $i=n$ are easily verified. For $1<i<n$ we check the equality
\[\nu^n_{G,x,\alpha}((i,i+1)\psi)=(i,i+1)\nu^n_{G,x,\alpha}(\psi),\] for all $\psi\in C^n(G,k\bar{x}).$
Let $g_1,\ldots, g_n\in G.$ Then
\begin{align*}
\nu^n_{G,x,\alpha}((i,i+1)\psi)(\overline{g_1}\otimes\ldots \otimes\overline{g_n})&=((i,i+1)\psi)(g_1,\ldots, g_n)\cdot \overline{g_1}\cdot\ldots \cdot\overline{g_n}\\
&=-\psi(g_1,\ldots, g_{i-1}g_i,g_i^{-1},g_ig_{i+1},\ldots, g_n)\cdot \overline{g_1}\cdot\ldots \cdot\overline{g_n}.
\end{align*}
On the other hand we get:
\begin{align*}
&((i,i+1)\nu^n_{G,x,\alpha})(\overline{g_1}\otimes\ldots\otimes \overline{g_n})\\
&=-\nu^n_{G,x,\alpha}(\psi)(\overline{g_1}\otimes\ldots \otimes \overline{ g_{i-1}}\cdot \overline{g_i}\otimes \overline{g_i}^{^{-1}}\otimes\overline{g_i} \cdot \overline{g_{i+1}}\otimes\ldots \otimes \overline{ g_n})\\
&=-\psi(g_1,\ldots, g_{i-1}g_i,g_i^{-1},g_ig_{i+1},\ldots, g_n)\cdot \overline{g_1}\cdot\ldots \cdot\overline{g_n}\\
&\times \alpha(g_{i-1},g_i)\alpha(g_i,g_i^{-1})^{-1}\alpha(g_i,g_{i+1})\alpha(g_{i-1}g_i,g_i^{-1})\alpha(g_i,g_{i+1})^{-1}\\
&=-\psi(g_1,\ldots, g_{i-1}g_i,g_i^{-1},g_ig_{i+1},\ldots, g_n)\cdot \overline{g_1}\cdot\ldots \cdot\overline{g_n}.
\end{align*}

$c)$ Let $g_1,\ldots, g_n\in G.$ For $i=1$ we have
\begin{align*}
&\nu^n_{G,x}((1,2)\psi)(g_1\otimes\ldots\otimes g_n)\\
&=\sum_{j=1}^{n_x}((1,2)\psi)(h_{j,1},\ldots,h_{j,n})x^{-1} x_jg_1\ldots g_n\\
&=-\sum_{j=1}^{n_x}(h_{j,1}\psi(h_{j,1}^{-1},h_{j,1}h_{j,2},\ldots,h_{j,n})x^{-1})x_jg_1\ldots g_n,
\end{align*}
where $h_{j,1},\ldots, h_{j,n}$ are determined as in \ref{trivial_def}.
Here, for later use, we set $$B=\sum_{j=1}^{n_x}(h_{j,1}\psi(h_{j,1}^{-1},h_{j,1}h_{j,2},\ldots,h_{j,n})x^{-1})x_j.$$
On the other hand have
\begin{align*}
&((1,2)\nu^n_{G,x}(\psi))(g_1\otimes\ldots\otimes g_n)\\
&=-g_1\nu^n_{G,x}(\psi)(g_1^{-1}\otimes g_1g_2\otimes\ldots \otimes g_n)\\
&=-\sum_{i=1}^{n_x}g_1(\psi(h_{i,1}',\ldots,h'_{i,n})x^{-1})x_ig_2\ldots g_n.
\end{align*}
We need to verify the equality
\begin{equation} \label{eqn51}
B=\sum_{i=1}^{n_x}g_1(\psi(h_{i,1}',\ldots,h'_{i,n})x^{-1})x_ig_1^{-1},
\end{equation}

where, for $i\in \{1,\ldots,n_x\},$ the elements $h_{i,1}',\ldots,h'_{i,n}\in C_G(x)$ are determined by the relations
\begin{align*}
\gamma_{i,x}\ g_1^{-1}&=h'_{i,1}\gamma_{t_i^1,x},\\
\gamma_{t_i^1,x}\ g_1g_2&=h'_{i,2}\gamma_{t_i^2,x},\\
&\vdots\\
\gamma_{t_i^{n-1},x}\ g_n&=h'_{i,n}\gamma_{t_i^n,x}.\\
\end{align*}
Since when $j$ runs through the elements of $\{1,\ldots, n_x\},$ $s_j^1$ runs through the elements of the same set, we obtain that the right hand member
of the equality (\ref{eqn51}) is
\begin{align*}
\sum_{j=1}^{n_x}g_1(\psi(h_{s_j^1,1}',\ldots,h'_{s_j^1,n})x^{-1})x_{s_j^1}g_1^{-1},
\end{align*}
by replacing $i:=s_j^1.$
Next,  using \ref{trivial_def} again we get
 \begin{align*}
 \gamma_{s_j^1,x}g_1^{-1}=h^{-1}_{j,1}\gamma_{j,x}&\Rightarrow h'_{s_j^1,1}=h^{-1}_{j,1},\\
 \gamma_{j,x}g_1g_2=h_{j,1}h_{j,2}\gamma_{s_j^2,x}&\Rightarrow h'_{s_j^1,2}=h_{j,1}h_{j,2},\\
 &\vdots\\
 \gamma_{s_j^{n-1},x}g_n=h_{j,n}\gamma_{s_j^n,x}&\Rightarrow h'_{s_j^1,n}=h_{j,n}.\\
 \end{align*}
 Moreover
 \[x_{s_j^1}=\gamma_{s_j^1,x}^{-1}\ x\ \gamma_{s_j^1,x}=g^{-1}_1\ \gamma_{j,x}^{-1}\ h_{j,1}\ x\ h_{j,1}^{-1}\ \gamma_{j,x}\ g_1=g_1^{-1}\ x_j\ g_1\]
and then
\begin{align*}
&\sum_{j=1}^{n_x}g_1(\psi(h_{s_j^1,1}',\ldots,h'_{s_j^1,n})x^{-1})x_{s_j^1}g_1^{-1}\\
&=\sum_{j=1}^{n_x}g_1(\underbrace{\psi(h^{-1}_{j,1},h_{j,1}h_{j,2},\ldots, h_{j,n})x^{-1}}_{\in k})g_1^{-1}x_jg_1g_1^{-1}\\
&=\sum_{j=1}^{n_x}\psi(h^{-1}_{j,1},h_{j,1}h_{j,2},\ldots, h_{j,n})x^{-1}x_j=B.
\end{align*}
Similar arguments work in any of the cases $1<i\leq n.$
\end{proof}
\end{proposition}

\begin{proof}\textbf{(of Theorem \ref{thm11})}

$a)$ Let $x\in Z(G).$ According to Lemma \ref{lemma4.5} assertion $b),$ $\nu^*_{G,x,\alpha}$ commutes with the differentials. Consequently, Proposition \ref{proposition5.1} implies
\[\Ima\left(\left.\nu^n_{G,x,\alpha}\right|_{CS^n(G,k\bar{x})}\right)\subseteq CS^n(k_{\alpha}G,k_{\alpha}G),\] where
$\left.\nu^n_{G,x,\alpha}\right|_{CS^n(G,k\bar{x})}$ is the restriction of $\nu^n_{G,x,\alpha}$ to $CS^n(G,k\bar{x}).$ It follows that  the map
\[\nu^{n,S}_{G,x,\alpha}([\psi]):=\left[\left.\nu^n_{G,x,\alpha}\right|_{CS^n(G,k\bar{x})}(\psi)\right],\]
for any $\psi\in \Ker \partial^{n}\cap CS^n(G,k\bar{x})$, is well defined. The analogous statements are fulfilled by the first assertions of Proposition \ref{proposition5.1} and of Lemma \ref{lemma4.5} respectively.  The fact that $\bigoplus_{x\in Z(G)}\nu_{G,x,\alpha}^{n,S}$ is an embedding, admitting $\bigoplus_{x\in Z(G)}\pi_{G,x,\alpha}^{n,S}$ as its left inverse, is a consequence of Lemma \ref{lemma4.5} statement $c).$
The proof of the assertion $b)$ is similar to that of $a)$, taking into account the remaining affirmations of Lemma \ref{lemma4.5} and of Proposition \ref{proposition5.1}.
\end{proof}

\section{Proof of Proposition \ref{prop62} and Theorem \ref{thm13}}\label{sec6}

\begin{proof}\textbf{(of Propostion \ref{prop62})}
By \ref{subsec61} and some arguments of \cite[Proposition 5.3]{Si1} it is enough to verify that $\beta_G^n(\varphi)\in CS^{n+1}(G,A)$ for any $\varphi\in CS^n(G,C)$.

Let $\varphi\in CS^n(G,C)$ and $(i,i+1)\in\Sigma_{n+1}, i\in\{1,\ldots,n+1\}$.

For $i=1$ we obtain
\begin{align*}
((1,& 2)\beta_G^n(\varphi))(g_1,\ldots,g_{n+1})\\
&=-g_1(r_{\iota}(\partial_B^n(s_{\pi}\circ \varphi)(g_1^{-1},g_1g_2,\ldots, g_{n+1})))\\
&=r_{\iota}\left[-g_1\left(g_1^{-1}s_{\pi}(\varphi(g_1g_2,g_3,\ldots,g_{n+1}))-s_{\pi}(\varphi(g_2,g_3,\ldots,g_{n+1}))+s_{\pi}(\varphi(g_1^{-1},g_1g_2g_3,\ldots,g_{n+1}))+\right.\right.\\
&\left.\left.\cdots+(-1)^ns_{\pi}(\varphi (g_1^{-1}, g_1g_2,\ldots, g_ng_{n+1}))+(-1)^{n+1}s_{\pi}( \varphi(g_1^{-1},g_1g_2,\ldots,g_n))\right)\right]\\
&=r_{\iota}\left[s_{\pi}(-\varphi(g_1g_2,g_3,\ldots,g_{n+1}))-s_{\pi}(-g_1\varphi(g_2,g_3,\ldots,g_{n+1}))+s_{\pi}(-g_1\varphi(g_1^{-1},g_1g_2g_3,\ldots,g_{n+1}))+\right.\\&
\left.\cdots+(-1)^n s_{\pi}(-g_1\varphi (g_1^{-1}, g_1g_2,\ldots, g_ng_{n+1}))+(-1)^{n+1}s_{\pi}(-g_1\varphi(g_1^{-1},g_1g_2,\ldots,g_n))\right]\\
&=r_{\iota}\left[g_1(s_{\pi}\circ\varphi)(g_2,g_3,\ldots,g_{n+1})-(s_{\pi}\circ\varphi)(g_1g_2,g_3,\ldots,g_{n+1}))+(s_{\pi}\circ\varphi)(g_1,g_2g_3,\ldots,g_{n+1})+\right.\\&
\left.\cdots+(-1)^n (s_{\pi}\circ\varphi)( g_1,g_2,\ldots, g_ng_{n+1}))+(-1)^{n+1}(s_{\pi}\circ \varphi)(g_1,g_2,\ldots,g_n))\right]\\&
=r_{\iota}\left(\partial_B^n(s_{\pi}\circ\varphi)(g_1,\ldots,g_{n+1})\right)\\&
=\beta_G^n(\varphi)(g_1,\ldots,g_{n+1}),
\end{align*}
where in the  third equality we used that $s_\pi$ is a symmetric section, compatible with the actions; in the fourth equality we reversed the first and second term and we used that $\varphi$ is a symmetric $n$-cochain from  $CS^n(G,C).$
For $i=n+1$ we have
\begin{align*}
((n,& n+1)\beta_G^n(\varphi))(g_1,\ldots,g_{n+1})\\
&=-r_{\iota}[\partial_B^n(s_{\pi}\circ \varphi)(g_1,g_2,\ldots, g_ng_{n+1},g_{n+1}^{-1}))]\\
&=-r_{\iota}\left[g_1(s_{\pi}(\varphi(g_2,g_3,\ldots,g_ng_{n+1},g_{n+1}^{-1}))-s_{\pi}(\varphi(g_1g_2,g_3,\ldots,g_ng_{n+1},g_{n+1}^{-1}))+\right.\\
&\left.\cdots+(-1)^ns_{\pi}(\varphi (g_1, g_2,\ldots, g_n))+(-1)^{n+1}s_{\pi}( \varphi(g_1,g_2,\ldots,g_ng_{n+1}))\right]\\
&=r_{\iota}\left[s_{\pi}(-g_1\varphi(g_2,g_3,\ldots,g_ng_{n+1},g_{n+1}^{-1}))+s_{\pi}(-\varphi(g_1g_2,g_3,\ldots,g_ng_{n+1},g_{n+1}^{-1}))+\right.\\&
\left.\cdots-(-1)^n s_{\pi}(\varphi (g_1,g_2,\ldots, g_n))-(-1)^{n+1}s_{\pi}(\varphi(g_1,g_2,\ldots,g_ng_{n+1}))\right]\\
&=r_{\iota}\left[g_1(s_{\pi}\circ\varphi)(g_2,g_3,\ldots,g_{n+1})-(s_{\pi}\circ\varphi)(g_1g_2,g_3,\ldots,g_{n+1}))+\right.\\&
\left.\cdots+(-1)^n (s_{\pi}\circ\varphi)( g_1,g_2,\ldots, g_ng_{n+1}))+(-1)^{n+1}(s_{\pi}\circ \varphi)(g_1,g_2,\ldots,g_n))\right]\\&
=r_{\iota}\left(\partial_B^n(s_{\pi}\circ\varphi)(g_1,\ldots,g_{n+1})\right)\\&
=\beta_G^n(\varphi)(g_1,\ldots,g_{n+1}),
\end{align*}
where  in the fourth equality we reversed the last two terms and for the other terms we used $\varphi\in  CS^n(G,C).$

The other cases, $1<i<n+1$, are verified similarly.
\end{proof}
\begin{proof}\textbf{(of Theorem \ref{thm13})}
We show  that the diagram is  commutative in terms of cochain complexes. Let $h_1,\ldots, h_{n+1}\in C_G(x)$ and $\varphi\in CS^n(k_3G,k_3G)$. We fix the notations
\begin{equation}\label{eq631}
\begin{split}
\varphi(h_2\otimes\ldots \otimes h_{n+1})h_{n+1}^{-1}\ldots h_2^{-1}&=\sum_{y\in G}a_{0,\varphi,y}\ y,\\
\varphi(h_1\otimes\ldots \otimes h_ih_{i+1}\otimes h_{n+1})h_{n+1}^{-1}\ldots h_1^{-1}&=\sum_{y\in G}a_{i,\varphi,y}\ y,\quad i\in\{1,\ldots,n\},\\
\varphi(h_1\otimes\ldots \otimes h_{n})h_{n}^{-1}\ldots h_1^{-1}&=\sum_{y\in G}a_{n+1,\varphi,y}\ y,
\end{split}
\end{equation}
We have
\begin{align*}
\beta_{C_G(x)}^{n,S}&(\pi_{G,x}^{n,S}(\varphi))(h_1,\ldots, h_{n+1})=(r_{\iota_x}\circ\partial_{k_2}^n(s_{\pi_x}\circ\pi_{G,x}^n(\varphi)))(h_1,\ldots, h_{n+1})\\
&=r_{\iota_x}(h_1s_{\pi_x}(\pi_{G,x}^n(\varphi)(h_2,\ldots,h_{n+1}))+\sum_{i=1}^n(-1)^i s_{\pi_x}(\pi_{G,x}^n(\varphi)(h_1,\ldots,h_ih_{i+1},\ldots,h_{n+1}))\\
& +(-1)^{n+1}s_{\pi_x}(\pi_{G,x}^n(\varphi)(h_1,\ldots,h_n)) )\\
&=r_{\iota_x}\left(\sum_{i=0}^{n+1}(-1)^is_{\pi_x}(a_{i,\varphi,x}\ x)\right)\\
&=r_{\iota}\left(\sum_{i=0}^{n+1}(-1)^is_{\pi}(a_{i,\varphi,x})\right) x.
\end{align*}

On the other hand, by \ref{prima'}
$$\pi_{G,x}^{n+1,S}(\mathbb{B}_G^n(\varphi))(h_1,\ldots,h_{n+1})=a_{1,x}\ x,$$
where $a_{1,x}$ is the coefficient of $x$ in
\begin{align*}
&\quad (r_{\iota_G}\circ \delta_{k_2G}^n(s_{\pi_G}\circ\varphi))(h_1\otimes\ldots\otimes h_{n+1})h_{n+1}^{-1}\ldots h_1^{-1}\\
&=r_{\iota_G}( h_1(s_{\pi_G} \circ\varphi)(h_2,\ldots,h_{n+1})+\sum_{i=1}^n(-1)^i (s_{\pi_G}\circ\varphi)(h_1,\ldots,h_ih_{i+1},\ldots,h_{n+1})\\
&+(-1)^{n+1}(s_{\pi_G}\circ\varphi)(h_1,\ldots,h_n) )h_{n+1}^{-1}\ldots h_1^{-1}\\
&=r_{\iota_G}\left( s_{\pi_G}\left(\sum_{y\in G}h_1(a_{0,\varphi,y}\ y)h_1^{-1}\right)+\sum_{i=1}^n(-1)^i s_{\pi_G}\left(\sum_{y\in G} a_{i,\varphi,y}\ y\right)+(-1)^{n+1}s_{\pi_G}\left(\sum_{y\in G}a_{n+1,\varphi,y}\ y\right) \right) \\ &=r_{\iota_G}\left( \sum_{y\in G}\left( s_{\pi}(a_{0,\varphi,h_1^{-1}yh_1})+\sum_{i=1}^{n+1}(-1)^i s_{\pi}(a_{i,\varphi,y})\right)\ y \right)
\\ &=r_{\iota}\left( \sum_{y\in G}\left( s_{\pi}(a_{0,\varphi,h_1^{-1}yh_1})+\sum_{i=1}^{n+1}(-1)^i s_{\pi}(a_{i,\varphi,y})\right)\right)\ y ,
\end{align*}
where the above equalities are true by \ref{subsec21}, (\ref{eq631}) and the  definitions of $r_{\iota_G},s_{\pi_G}$.
It follows that $$a_{1,x}=r_{\iota}\left( s_{\pi}(a_{0,\varphi,h_1^{-1}xh_1})+\sum_{i=1}^{n+1}(-1)^i s_{\pi}(a_{i,\varphi,x})\right)=r_{\iota}\left(\sum_{i=0}^{n+1}(-1)^i s_{\pi}(a_{i,\varphi,x})\right)$$
since $h_1\in C_G(x)$.
\end{proof}

\begin{remark}\label{rem'}
\begin{itemize}
\item[a)] The construction of $\mathbb{B}_G^*$  is obtained by using a trick. We reinterpret the map from Proposition \ref{prop62} and an isomorphism
$$F^*:\HS^*(G,k_{\alpha}G)\rightarrow\HHS^*(k_{\alpha}G)$$
applied for $\alpha=\one_k$; here $G$ acts on $k_{\alpha}G$ by ${}^gm=\bar{g}\cdot m\cdot \bar{g}^{-1}$ for any $g\in G$ and $m\in k_{\alpha} G$. This isomorphism $F^n:CS^n(G,k_{\alpha}G) \rightarrow CS^n(k_{\alpha}G, k_{\alpha} G)$ can be easily described by
$$F^n(\psi)(\bar{g_1}\otimes\ldots\otimes \overline{g_n})=\psi(g_1,\ldots,g_n)\cdot \overline{g_1}\cdot\ldots\cdot \overline{g_n}$$
for any $\psi\in CS^n(G,k_{\alpha}G), g_1,\ldots, g_n\in G$. This isomorphism is obtained as a consequence of \cite[Theorem 3.2.4]{GG} where the general case of Hopf crossed products is investigated. As a matter of fact the same trick is applied to obtain the action of the symmetric group on the Hochschild cochain complex in (\ref{action_on_chain}).
\end{itemize}
\end{remark}

\ack{We thank  Professor J. A. Guccione for some email clarifications with respect to the isomorphism in Remark \ref{rem'}. We are also grateful to an unknown referee for his/her suggestions
which significantly improved the first version of this article.

\end{document}